\documentclass[12pt]{article}

\usepackage{amsmath}
\usepackage{amsfonts}
\usepackage{amssymb}
\usepackage{mathrsfs}
\usepackage{graphicx}

\newtheorem{theorem}{Theorem}

\newtheorem{definition}{Definition}
\newtheorem{example}{Example}
\newtheorem{lemma}{Lemma}

\newtheorem{proposition}{Proposition}
\newtheorem{remark}{Remark}

\newenvironment{proof}[1][Proof]{\noindent\textbf{#1.} }{\ \rule{0.5em}{0.5em} \\}

\def\t{\tau}
\def\R{\mathbb{R}}
\def\N{\mathbb{N}}
\newcommand{\fonction}[5]{\begin{array}[t]{lrcl}#1 :&#2 &\longrightarrow &#3\\&#4& \longmapsto &#5 \end{array}}

\begin{document}

\author{Lo\"ic Bourdin \\ Institut de Recherche XLIM \\ D\'epartement de Math\'ematiques et d'Informatique DMI \\ Universit\'e de Limoges\\UMR CNRS 7252, Limoges, France \\ \textit{e-mail}: loic.bourdin@unilim.fr
\and Dariusz Idczak\\Faculty of Mathematics and Computer Science\\University of Lodz\\Banacha 22, 90-238 Lodz, Poland\\\textit{e-mail}: idczak@math.uni.lodz.pl}
\title{A fractional fundamental lemma and a fractional integration by parts formula -- Applications to critical points of Bolza functionals and to linear boundary value problems}
\date{}
\maketitle

\begin{abstract}
In this paper we first identify some integrability and regularity issues that frequently occur in fractional calculus of variations. In particular, it is well-known that Riemann-Liouville derivatives make boundary singularities emerge. The major aim of this paper is to provide a framework ensuring the validity of the fractional Euler-Lagrange equation in the case of a Riemann-Liouville derivative of order $\alpha \in (0,1)$. For this purpose, we consider the set of functions possessing $p$-integrable Riemann-Liouville derivatives and we introduce a class of quasi-polynomially controlled growth Lagrangian.

In the first part of the paper we prove a new fractional fundamental (du Bois-Reymond) lemma and a new fractional integration by parts formula involving boundary terms. The proof of the second result is based on an integral representation of functions possessing Riemann-Liouville derivatives. In the second part of the paper we give not only a necessary optimality condition of Euler-Lagrange type for fractional Bolza functionals, but also necessary optimality boundary conditions. Finally we give an additional application of our results: we prove an existence result for solutions of linear fractional boundary value problems. This last result is based on a Hilbert structure and the classical Stampacchia theorem.
\end{abstract}



\tableofcontents

\section{Introduction}
We first introduce some notations available in the whole paper. Let $m$ be a nonzero integer and let $\Vert \cdot \Vert$ be the Euclidean norm of $\R^m$. Let $(a,b) \in \R^2$ such that $a<b$. For any $1 \leq p \leq \infty$, by $\mathrm{L}^p := \mathrm{L}^p((a,b),\R^m)$ (resp. $\mathrm{W}^{1,p} := \mathrm{W}^{1,p}((a,b),\R^m)$) we denote the usual $p$-Lebesgue space (resp. $p$-Sobolev space) endowed with its usual norm $\Vert \cdot \Vert_{\mathrm{L}^p}$ (resp. $\Vert \cdot \Vert_{\mathrm{W}^{1,p}}$) and by $p' := \frac{p}{p-1}$ we denote the usual adjoint of $p$. By $\mathrm{AC} := \mathrm{AC} ( [a,b],\mathbb{R}^{m})$ we denote the classical space of absolutely continuous functions on $[a,b]$. Finally, by $\mathscr{C}_{c}^{\infty}:=\mathscr{C}_{c}^{\infty}( [a,b],\mathbb{R}^{m})$ we denote the
classical space of smooth functions on $[a,b]$ with compact supports contained in $(a,b)$.

\paragraph*{Fractional calculus of variations, a well-known strategy.} Fractional calculus is the mathematical field that deals with the generalization of the classical notions of integral and derivative to any positive real order, see \cite{KilSri,Samko} for a detailed study. It has numerous applications in various areas of science, mainly to study anomalous processes. Many physical applications can be found in the monograph \cite{Hermann}. 

The introduction of the fractional operators in calculus of variations is due to F.~Riewe in \cite{Riewe} where he aimed at giving a fractional variational structure to some non conservative systems. Since \cite{Riewe}, numerous necessary optimality conditions of Euler-Lagrange type were obtained for functionals depending on different notions of fractional integrals and/or derivatives (Riemann-Liouville, Caputo, etc.). These fractional functionals have different forms and are defined on different function spaces with different initial/boundary conditions. It is not our aim to give a state of the art here. We refer to \cite{Agrawal,Almeida2,Almeida1,Almeida3,BM,Cresson,NabuTorr,Odz1,Odz2} and to the monographs \cite{Klimek,MalTor} for numerous examples.

The above results are mostly based on a common strategy that we can briefly sum up as follows. Let us consider a functional of type:
$$ \fonction{\Phi}{\mathrm{E}}{\R}{q}{\displaystyle \int_a^b L(\t,q(\t),(\mathrm{O}^\alpha_{a+} q) (\t) ) \; d\t ,}$$
where $\mathrm{E}$ is a function space, $L : (t,x,v) \in [a,b] \times \R^m \times \R^{m} \longmapsto L(t,x,v) \in \R$ is a regular Lagrangian and $\mathrm{O}^\alpha_{a+}$ is a left-sided fractional operator (Riemann-Liouville integral or Caputo derivative for example) of order $\alpha \in (0,1)$. The first variation $\delta\Phi(q,h)$ of the functional $\Phi$ at a point $q \in \mathrm{E}$ in a direction $h \in \mathrm{E}$ such that $h(a)=h(b)=0$ is given by:
\begin{equation*}
\delta\Phi(q,h)= \int_a^b  L_x (\t,q(\t),(\mathrm{O}^\alpha_{a+} q)(\t) ) \cdot h(\t) + L_v (\t,q(\t),(\mathrm{O}^\alpha_{a+} q)(\t) ) \cdot (\mathrm{O}^\alpha_{a+} h)(\t) \; d\t .
\end{equation*}
Using a fractional integration by parts formula on the second term of the integrand, one obtains:
$$ \delta\Phi(q,h)= \int_a^b \Big[ L_x (\t, q(\t),(\mathrm{O}^\alpha_{a+} q)(\t) ) + \mathrm{O}^\alpha_{b-} ( L_v (\cdot,q,\mathrm{O}^\alpha_{a+} q ) ) (\t) \Big] \cdot h(\t) \; d\t ,$$
where $\mathrm{O}^\alpha_{b-}$ is the right-sided fractional operator associated with $\mathrm{O}^\alpha_{a+}$. Finally, the fundamental lemma (recalled in Lemma~\ref{CDBR1}) concludes that a necessary condition for critical points of the functional $\Phi$ is given by the following fractional Euler-Lagrange equation:
$$ L_x (t,q(t),(\mathrm{O}^\alpha_{a+} q)(t) ) + \mathrm{O}^\alpha_{b-} ( L_v (\cdot,q,\mathrm{O}^\alpha_{a+} q) )(t) =0, \quad t \in [a,b] \text{ a.e.} $$

\paragraph*{Some integrability and regularity issues.} Note that some integrability and regularity issues occur in the above strategy, in particular in the case where $\mathrm{O}^\alpha_{a+}$ is the Riemann-Liouville derivative. For instance:
\begin{enumerate}
\item The functional $\Phi$ is well-defined provided that the integrability of $ L(\cdot,q ,\mathrm{O}^\alpha_{a+} q )$ is satisfied for any $q \in \mathrm{E}$. For example, in the case where $q(a) \neq 0$ and $\mathrm{O}^\alpha_{a+}$ is the Riemann-Liouville derivative, $\mathrm{O}^\alpha_{a+} q$ admits a singularity at $t = a$, even if $q$ is a smooth function. Hence, the integrability of $ L(\cdot,q ,\mathrm{O}^\alpha_{a+} q )$ is not ensured, even if $L$ is smooth.
\item The same issue occurs in the derivation of the first variation $\delta\Phi(q,h)$.
\item The use of a fractional integration by parts formula requires that $L_v (\cdot,q ,\mathrm{O}^\alpha_{a+} q ) $ satisfies some regularity hypotheses. In particular, one has to ensure that $\mathrm{O}^\alpha_{b-} ( L_v (\cdot,q,\mathrm{O}^\alpha_{a+} q) )$ is well-defined and has to prove that the scalar product $\mathrm{O}^\alpha_{b-} ( L_v (\cdot,q,\mathrm{O}^\alpha_{a+} q) ) \cdot h$ is integrable.
\item In the literature, necessary optimality conditions are mostly concerned with functions $q$ satisfying fixed boundary conditions of type $q(a)= q_a$ and $q(b)=q_b$. Consequently, variations $h$ are mostly assumed to vanish at $t = a$ and $t =b$ and then no boundary term emerges in the fractional integration by parts formula. In the case where authors would like to treat functions $q$ with free boundary values and to obtain necessary optimality boundary conditions, variations $h$ that do not vanish at $t = a$ and $t =b$ have to be considered. In such a case, fractional boundary terms would emerge in the fractional integration by parts formula provided that $L_v (\cdot,q ,\mathrm{O}^\alpha_{a+} q ) $ satisfies some regularity hypotheses.
\end{enumerate}
As a consequence, the choice of the function space $\mathrm{E}$, with respect to the left-sided fractional operator $\mathrm{O}^\alpha_{a+}$ considered, is crucial in order to valid each step of the above strategy. The choice of variations $h$ is also capital.  

\paragraph*{Contributions of this paper.} It is well-known that Riemann-Liouville derivatives make boundary singularities emerge. In this paper, our major aim is to provide a framework ensuring the validity of each step of the above strategy in the case where $\mathrm{O}^\alpha_{a+}$ is the left-sided fractional Riemann-Liouville derivative of order $\alpha \in (0,1)$. For this purpose:
\begin{enumerate}
\item We consider the function space $\mathrm{E} = \mathrm{AC}^{\alpha,p}_{a+}$ being the space of functions $q$ possessing $p$-integrable left-sided fractional Riemann-Liouville derivatives of order $\alpha \in (0,1)$, see Definitions~\ref{defdef1}, \ref{defdef2}, \ref{defdef3}, \ref{defdef4} in Sections~\ref{section2}, \ref{section3}.
\item We introduce a class of quasi-polynomially controlled growth Lagrangian $L$ ensuring the integrability of the integrands considered in the previous paragraphs, see Definitions~\ref{defdef5}, \ref{defdef6} in Section~\ref{section4}.
\end{enumerate}
Using integral representations of functions $q \in \mathrm{AC}^{\alpha,p}_{a+}$ (given in Proposition~\ref{dziewiec}), we prove a new fractional integration by parts formula involving boundary values, see Theorem~\ref{cze3}. Studying fractional Bolza functionals, this new formula will allow us to derive not only a necessary optimality condition of Euler-Lagrange type but also necessary optimality boundary conditions, see Theorem~\ref{TEL1}.

To conclude this paragraph, let us recall two versions of the classical fundamental lemma:

\begin{lemma}[classical fundamental lemma, version 1]\label{CDBR1}
If $q \in \mathrm{L}^{1}$ and
$$ \displaystyle \int_{a}^{b} q (\t) \cdot h(\t) \; d\t = 0 $$
for any function $h\in \mathscr{C}_{c}^{\infty}$, then $q = 0$ a.e.
\end{lemma}

\begin{lemma}[classical fundamental lemma, version 2]\label{CDBR2}
If $q_{1}$, $q_{2}\in \mathrm{L}^{1}$ and
$$ \displaystyle \int_{a}^{b} q_{1}(\t) \cdot h(\t) + q_{2}(\t) \cdot h^{\prime}(\t) \; d\t=0 $$
for any function $h\in \mathscr{C}_{c}^{\infty}$, then $q_{2} \in \mathrm{AC}$
(more precisely, it has an absolutely continuous representative) and $q_{2}^{\prime}(t)=q_{1}(t)$,  $t\in  [a,b] \text{ a.e.}$
\end{lemma}
In the previous summarized strategy, a fractional integration by parts formula is used and Lemma~\ref{CDBR1} concludes. Recall that the use of a fractional integration by parts formula requires regularity hypotheses on $L_v (\cdot,q ,\mathrm{O}^\alpha_{a+} q ) $ and the existence of $\mathrm{O}^\alpha_{b-} ( L_v (\cdot,q,\mathrm{O}^\alpha_{a+} q) )$. Note that Lemma~\ref{CDBR2} is more interesting since we get a gain of regularity on $q_2$ and the existence of its derivative. In this paper, we prove a fractional variant of Lemma~\ref{CDBR2} (see Theorem~\ref{GDBR}) proving the regularity of $L_v (\cdot,q ,\mathrm{O}^\alpha_{a+} q ) $ and the existence of $\mathrm{O}^\alpha_{b-} ( L_v (\cdot,q,\mathrm{O}^\alpha_{a+} q) )$.

\paragraph*{Another application: an existence result.}
In \cite{MMRK}, the authors obtain the existence of extremals in the case of a fractional functional associated with a quadratic Lagrangian $L$. In \cite{Bourdin}, an existence result for solutions of general fractional Euler-Lagrange equations is investigated: the existence of a minimizer is obtained under suitable assumptions of regularity, coercivity and convexity. Similar methods were developed in \cite{Bourdin1,Bourdin2} in order to obtain existence results for minimizers of functionals of different forms and depending on different notions of fractional derivatives.

In this paper, we give an additional application of the fractional fundamental lemma (obtained in Theorem~\ref{GDBR}) and of the fractional integration by parts formula (obtained in Theorem~\ref{cze3}): we prove an existence result for solutions of linear fractional boundary value problems, see Theorem~\ref{thmexistence}. Our method is based on the Hilbert structure of $\mathrm{AC}^{\alpha,2}_{a+}$ and the classical Stampacchia theorem.

\paragraph*{Organization of the paper.}
The paper is organized as follows. In Section~\ref{section2} we give some basic recalls on fractional calculus and we prove a new fractional fundamental lemma (Theorem~\ref{GDBR}). In Section~\ref{section3}, from an integral representation of functions possessing fractional Riemann-Liouville derivatives, we derive a new fractional integration by parts formula (Theorem~\ref{cze3}). The two last sections are devoted to applications of the previous results to necessary optimality conditions for fractional Bolza functionals (Section~\ref{section4}, Theorem~\ref{TEL1}) as well as to existence of solutions for some linear boundary value problems (Section~\ref{section5}, Theorem~\ref{thmexistence}).

\section{A fractional fundamental lemma}\label{section2}

The aim of this section is to prove a fractional counterpart of Lemma \ref{CDBR2}. Section~\ref{section21} gives some basic recalls on fractional calculus and the main result is derived in Section~\ref{section22}.

\subsection{Preliminaries on fractional calculus}\label{section21}

By the left- and right-sided fractional Riemann-Liouville integrals of order $\alpha > 0$ of $q \in \mathrm{L}^{1}$ we mean the following functions
$$ (I_{a+}^{\alpha} q)(t):=\frac{1}{\Gamma(\alpha)}\int_{a}^{t} \frac{q(\tau)}{(t-\tau)^{1-\alpha}} \; d\tau, \quad t\in  [a,b]\text{ a.e.,} $$
$$ (I_{b-}^{\alpha}q)(t):=\frac{1}{\Gamma(\alpha)}\int_{t}^{b} \frac{q(\tau)}{(\tau-t)^{1-\alpha}} \; d\tau, \quad t \in [a,b] \text{ a.e.} $$
Recall that $I_{a+}^{\alpha}$, $I_{b-}^{\alpha}$ are linear continuous operator from $\mathrm{L}^p$ to $\mathrm{L}^p$ for any $1 \leq p \leq \infty$, see \cite[Theorem~2.6 p.48]{Samko}. The two following propositions have been proved respectively in \cite[Equation~(2.21) p.34]{Samko} and \cite[Corollary of Theorem~3.5 p.67]{Samko}.

\begin{proposition}\label{thmcomposition}
If $\alpha_1$, $\alpha_2 >0$, then $(I^{\alpha_1}_{a+}  I^{\alpha_2}_{a+} q)(t) = (I^{\alpha_1 + \alpha_2}_{a+} q)(t)$, $t \in [a,b] \text{ a.e.}$ for any $q \in \mathrm{L}^{1}$.
\end{proposition}

\begin{proposition}\label{czint}
If $\alpha>0$, $1\leq p \leq \infty$, $1\leq r \leq\infty$ and
$\frac{1}{p}+\frac{1}{r}\leq1+\alpha$ (additionally, we assume that $p>1$ and
$r>1$ when $\frac{1}{p}+\frac{1}{r}=1+\alpha$), then
$$ \displaystyle \int_{a}^{b} (I_{a+}^{\alpha}q_1)(\t) \cdot q_2(\t) \; d\t = \int_{a}^{b} q_1(\t) \cdot (I_{b-}^{\alpha}q_2)(\t) \; d\t $$
for any $q_1 \in \mathrm{L}^{p}$, $q_2 \in \mathrm{L}^{r}$.
\end{proposition}
For the following proposition, we refer to \cite{Hardy} (Property~1), \cite[Theorem~4 p.575]{HardyLittlewood} (Property~2) and \cite[Theorem~3.6 p.67]{Samko} (Property~4). Note that Property~3 can be easily derived from Property~2, see also discussions in \cite[p.578]{HardyLittlewood} and \cite[Paragraph~3.3 p.91]{Samko}. 

\begin{proposition}\label{nose}
Let $\alpha \in (0,1)$, $1 \leq p \leq \infty$ and $q \in \mathrm{L}^p$. The following statements are satisfied:
\begin{enumerate}
\item if $0 < \alpha < 1 = p$ then $I^{\alpha}_{a+} q \in \mathrm{L}^r$ for every $1 \leq r < \frac{1}{1-\alpha}$;
\item if $0 < \alpha < \frac{1}{p} < 1$ then $I^{\alpha}_{a+} q \in \mathrm{L}^r$ for every $1 \leq r \leq \frac{p}{1-\alpha p}$; 
\item if $0 < \alpha = \frac{1}{p} < 1$ then $I^{\alpha}_{a+} q \in \mathrm{L}^r$ for every $1 \leq r < \infty$;
\item if $0 \leq \frac{1}{p} < \alpha < 1$ then $I^{\alpha}_{a+} q \in \mathrm{L}^\infty$.
\end{enumerate}
Actually, in the last case, $I^\alpha_{a+} q$ can be identified to a H\"olderian continuous function with exponent $\alpha - \frac{1}{p}$ vanishing at $t=a$.
\end{proposition}

\begin{remark}
Analogous of Propositions~\ref{thmcomposition} and \ref{nose} for the right-sided fractional integral hold true.
\end{remark}

\subsection{Main result}\label{section22}

Recall that a function $q : [a,b] \longrightarrow \R^m$ is absolutely continuous if and only if there exists a couple $(c,\varphi) \in \mathbb{R}^{m} \times \mathrm{L}^{1}$ such that
\begin{equation}\label{eqabscontclass}
q(t)= c+(I^1_{a+} \varphi )(t),\quad t\in [a,b].
\end{equation}
In this case, $c = q(a)$ and $\varphi(t) = q^{\prime}(t)$, $t\in [a,b] \text{ a.e.}$

\begin{remark}
A fractional counterpart of the integral representation \eqref{eqabscontclass} is given in Proposition~\ref{dziewiec}. 
\end{remark}
Recall the following definitions (see \cite[Definition~2.4 p.44]{Samko}):

\begin{definition}\label{defdef1}
We say that $q\in \mathrm{L}^{1}$ possesses a left-sided Riemann-Liouville derivative
$D_{a+}^{\alpha}q$ of order $\alpha\in(0,1)$ if the
function $I_{a+}^{1-\alpha}q$ has an absolutely continuous representative. In this case, $I_{a+}^{1-\alpha}q$ is identified to its absolutely continuous representative and $D_{a+}^{\alpha}q$ is defined by $D_{a+}^{\alpha}q := \frac{d}{dt}(I_{a+}^{1-\alpha}q)$.
\end{definition}

\begin{definition}\label{defdef2}
We say that $q\in \mathrm{L}^{1}$ possesses a right-sided
Riemann-Liouville derivative $D_{b-}^{\alpha}q$ of order $\alpha\in(0,1)$ if the function $I_{b-}^{1-\alpha}q$ has
an absolutely continuous representative. In this case, $I_{b-}^{1-\alpha}q$ is identified to its absolutely continuous representative and $D_{b-}^{\alpha}q$ is defined by $D_{b-}^{\alpha}q:=-\frac{d}{dt}(I_{b-}^{1-\alpha}q)$.
\end{definition}
Using the method of the proof presented in \cite[proof of Theorem 2]{Bourdin} we obtain the following fractional counterpart of Lemma \ref{CDBR2}.

\begin{theorem}[fractional fundamental lemma]\label{GDBR}
If $ \alpha \in (0,1)$, $q_{1}\in \mathrm{L}^{1}$, $q_{2}\in \mathrm{L}^{1}$ and
$$ \displaystyle\int_{a}^{b} q_1(\t) \cdot h(\t) + q_2(\t) \cdot (D^\alpha_{a+} h) (\t) \; d\t=0 $$
for any $h\in \mathscr{C}_{c}^{\infty}$, then $q_{2}$ possesses a right-sided fractional Riemann-Liouville derivative and
$(D_{b-}^{\alpha} q_{2})(t)= -q_{1}(t)$, $t\in [a,b]\text{ a.e.}$
\end{theorem}

\begin{proof}
For any $h\in \mathscr{C}_{c}^{\infty}$, $h(t)= (I_{a+}^{1}h^{\prime}) (t)$, $t \in [a,b]$. From Proposition~\ref{thmcomposition}, $(I^{1-\alpha}_{a+} h)(t) = (I^1_{a+}  I^{1-\alpha}_{a+} h^\prime)(t)$, $t \in [a,b]\text{ a.e.}$, where $I^{1-\alpha}_{a+} h^\prime \in \mathrm{L}^\infty \subset \mathrm{L}^1$. Thus, $I^{1-\alpha}_{a+} h$ has an absolutely continuous representative. Then, $h$ possesses a left-sided fractional Riemann-Liouville derivative given by $D^\alpha_{a+} h = I^{1-\alpha}_{a+} h^\prime \in \mathrm{L}^\infty$. Consequently, our assumption and Proposition~\ref{czint} imply
$$ \int_a^b q_1(\t) \cdot h(\t) + (I^{1-\alpha}_{b-} q_2)(\t) \cdot h^\prime (\t) \; d\t=0 $$
for any $h\in \mathscr{C}_{c}^{\infty}$. Lemma~\ref{CDBR2} concludes the proof.
\end{proof}
\begin{remark}
In the limit case of $\alpha=1$, the above lemma leads to the classical
fundamental lemma (Lemma~\ref{CDBR2}).
\end{remark}

\section{A fractional integration by parts formula}\label{section3}

As usually, let $I_{a+}^{\alpha}(\mathrm{L}^{p})$ and $I_{b-}^{\alpha}(\mathrm{L}^{p})$ respectively denote the ranges of the operators $I_{a+}^{\alpha}$, $I_{b-}^{\alpha}$ on $\mathrm{L}^p$ for any $1 \leq p \leq \infty$. 

\begin{remark}\label{rmk111}
In particular, if $q \in I_{a+}^{\alpha}(\mathrm{L}^{1})$ with $q = I^\alpha_{a+} \varphi$, $\varphi \in \mathrm{L}^1$, then $q$ possesses a left-sided Riemann-Liouville derivative given by $D_{a+}^{\alpha}q = \varphi$, see \cite[Theorem~2.4 p.44]{Samko}. Analogous result for the right-sided derivative holds true.
\end{remark}
From the previous remark and Proposition \ref{czint}, the following fractional integration by parts formula follows, see \cite[Corollary~2 p.46]{Samko}.

\begin{proposition}\label{czroz}
If $\alpha \in (0,1)$, $1\leq p \leq \infty$, $1\leq r \leq \infty$ and $\frac{1}{p}+\frac{1}{r}\leq1+\alpha$ (additionally, we assume that $p>1$ and $r>1$ when $\frac{1}{p}+\frac{1}{r}=1+\alpha$), then
$$ \displaystyle\int_{a}^{b} (D_{a+}^{\alpha}q_1)(\t) \cdot q_2(\t) \; d\t = \int_{a}^{b} q_1(\t) \cdot (D_{b-}^{\alpha}q_2)(\t) \; d\t $$
for any $q_1\in I_{a+}^{\alpha}(\mathrm{L}^{p})$, $q_2\in I_{b-}^{\alpha}(\mathrm{L}^{r})$.
\end{proposition}
The aim of this section is to extend Proposition~\ref{czroz} to all functions $q_1$, $q_2$ possessing fractional Riemann-Liouville derivatives. Our result will be valid under the additional assumptions $D^\alpha_{a+} q_1 \in \mathrm{L}^p$ and $D^\alpha_{b-} q_2 \in \mathrm{L}^r$ with $0 \leq \frac{1}{p} < \alpha < 1$ and $0 \leq \frac{1}{r} < \alpha < 1$.

An integral representation for functions possessing fractional Riemann-Liouville derivatives is given in Section~\ref{section31}. Section~\ref{section32} is devoted to the functional spaces $\mathrm{AC}^{\alpha,p}_{a+}$ and $\mathrm{AC}^{\alpha,p}_{b-}$. The main result is stated in Section~\ref{section33}.

\subsection{Integral representations}\label{section31}

Let us give a fractional counterpart of the integral representation \eqref{eqabscontclass}. 

\begin{proposition}[integral representation]\label{dziewiec}
Let $\alpha \in (0,1)$ and $q\in \mathrm{L}^{1}$. Then, $q$ has a left-sided Riemann-Liouville derivative $D_{a+}^{\alpha}q$ of
order $\alpha$ if and only if there exists a couple $(c,\varphi) \in \R^m \times \mathrm{L}^{1}$ such that
\begin{equation}\label{reprezentacja}
q(t)= \frac{1}{\Gamma(\alpha)}\frac{c}{(t-a)^{1-\alpha}} + (I^{\alpha}_{a+} \varphi )(t) ,\quad t \in [a,b]\text{ a.e.} 
\end{equation}
In this case, $c = (I_{a+}^{1-\alpha}q)(a)$ and $\varphi(t) = (D_{a+}^{\alpha}q)(t)$,  $t \in [a,b] \text{ a.e.}$
\end{proposition}

\begin{remark}
Note that the above fractional integral representation \eqref{reprezentacja} is given in \cite[Equality~(2.61) p.45]{Samko} as a necessary condition for functions possessing left-sided Riemann-Liouville derivative of order $\alpha \in (0,1)$ and the authors give a sketch of the proof, see the last sentence in the proof of \cite[Theorem~2.4 p.44]{Samko}. For sake of completeness, we give here a complete proof of this result and we prove that \eqref{reprezentacja} is also a sufficient condition. Moreover, the uniqueness of the couple $(c,\varphi)$ is stated.
\end{remark}

\begin{proof}[Proof of Proposition~\ref{dziewiec}]
Let us assume that $q \in \mathrm{L}^{1}$ has a left-sided Riemann-Liouville derivative $D_{a+}^{\alpha}q$. This
means that $I_{a+}^{1-\alpha}q$ is (identified to) an absolutely continuous function. From the integral representation \eqref{eqabscontclass}, there exists a couple $(c,\varphi) \in \R^m \times \mathrm{L}^{1}$ such that%
\begin{equation}\label{eq987}
(I_{a+}^{1-\alpha}q)(t)= c+(I^{1}_{a+} \varphi )(t),\quad t \in [a,b],
\end{equation}
with $(I_{a+}^{1-\alpha}q)(a) = c$ and $D_{a+}^{\alpha}q(t) = \frac{d}{dt} (I_{a+}^{1-\alpha}q)(t) = \varphi (t)$, $t \in [a,b]\text{ a.e.}$ From Proposition~\ref{thmcomposition} and applying $I_{a+}^{\alpha}$ on \eqref{eq987} we obtain
\begin{equation}\label{eq45654}
(I_{a+}^{1}q)(t)= (I^{\alpha}_{a+} c )(t) +(I^{1}_{a+}  I^{\alpha}_{a+} \varphi )(t),\quad t \in [a,b]\text{ a.e.}
\end{equation}
The result follows from the differentiation of \eqref{eq45654}. Now, let us assume that \eqref{reprezentacja} holds true. From Proposition~\ref{thmcomposition} and applying $I_{a+}^{1-\alpha}$ on \eqref{reprezentacja} we obtain
$$ (I_{a+}^{1-\alpha}q)(t)= c +(I^{1}_{a+} \varphi )(t),\quad t\in [a,b]\text{ a.e.} $$
and then, $I_{a+}^{1-\alpha}q$ has an absolutely continuous representative and $q$ has a left-sided Riemann-Liouville derivative $D_{a+}^{\alpha}q$. The proof is complete.
\end{proof}

Of course, in an analogous way one can prove:

\begin{proposition}
[integral representation]
Let $\alpha \in (0,1)$ and $q\in \mathrm{L}^{1}$. Then, $q$ has the right-sided
Riemann-Liouville derivative $D_{b-}^{\alpha}q$ of order $\alpha$ if and only if there exist a constant $d\in\mathbb{R}^{m}$
and a function $\psi \in \mathrm{L}^{1}$ such that
\begin{equation}\label{repr2}
q(t)= \frac{1}{\Gamma(\alpha)}\frac{d}{(b-t)^{1-\alpha}} + (I^{\alpha}_{b-} \psi )(t) ,\quad t\in [a,b]\text{ a.e.}
\end{equation}
In this case, $d = (I_{b-}^{1-\alpha}q)(b)$ and $\psi(t) = (D_{b-}^{\alpha}q)(t)$,  $t\in [a,b] \text{ a.e.}$
\end{proposition}

\subsection{Functional spaces $\mathrm{AC}_{a+}^{\alpha,p}$ and $\mathrm{AC}_{b-}^{\alpha,p}$}\label{section32}

By analogy with the set $\mathrm{AC}$ having the integral representation \eqref{eqabscontclass}, we introduce the following fractional counterpart $\mathrm{AC}_{a+}^{\alpha,p}$ as the set of all functions that have integral representation~\eqref{reprezentacja} with $\varphi\in \mathrm{L}^{p}$. We similarly introduce $\mathrm{AC}_{b-}^{\alpha,p}$ as the set of all functions that have integral representation~\eqref{repr2} with $\psi\in \mathrm{L}^{p}$.

\begin{definition}\label{defdef3}
For every $\alpha \in (0,1)$ and every $1 \leq p \leq \infty$, we denote by $\mathrm{AC}_{a+}^{\alpha,p}:=\mathrm{AC}_{a+}^{\alpha,p}([a,b],\mathbb{R}^{m})$ the set of all functions $q \in \mathrm{L}^1$ that have a left-sided Riemann-Liouville derivative $D_{a+}^{\alpha}q \in \mathrm{L}^p$. 
\end{definition}

\begin{definition}\label{defdef4}
For every $\alpha \in (0,1)$ and every $1 \leq p \leq \infty$, we denote by $\mathrm{AC}_{b-}^{\alpha,p}:=\mathrm{AC}_{b-}^{\alpha,p}([a,b],\mathbb{R}%
^{m})$ the set of all functions $q \in \mathrm{L}^1$ that have a right-sided Riemann-Liouville derivative $D_{b-}^{\alpha}q \in \mathrm{L}^p$. 
\end{definition}
The following proposition gives an integrability result for functions that belong to $\mathrm{AC}_{a+}^{\alpha,p}$ and $\mathrm{AC}_{b-}^{\alpha,p}$.

\begin{proposition}\label{remlr}
Let $\alpha \in (0,1)$ and $1 \leq p \leq \infty$. Then, the inclusions $\mathrm{AC}_{a+}^{\alpha,p} \subset \mathrm{L}^r$ and $\mathrm{AC}_{b-}^{\alpha,p} \subset \mathrm{L}^r$ hold for any $1 \leq r < \frac{1}{1-\alpha}$.
\end{proposition}

\begin{proof}
Recall that $I^\alpha_{a+} \varphi \in \mathrm{L}^r$ for any $\varphi \in \mathrm{L}^1$ and any $1 \leq r < \frac{1}{1-\alpha}$, see Proposition~\ref{nose}. Then the inclusion $\mathrm{AC}_{a+}^{\alpha,p} \subset \mathrm{L}^r$ easily follows from the integral representation~\eqref{reprezentacja}. The inclusion $\mathrm{AC}_{b-}^{\alpha,p} \subset \mathrm{L}^r$ can be analogously derived.
\end{proof}

The following proposition gives an explicit class of functions that belong to $\mathrm{AC}_{a+}^{\alpha,p}$ and $\mathrm{AC}_{b-}^{\alpha,p}$ in the case where $\alpha \in (0,1)$ and $1 \leq p < \frac{1}{\alpha}$.

\begin{proposition}\label{propabsdansabs}
Let $\alpha \in (0,1)$. The inclusion $\mathrm{AC} \subset \mathrm{AC}_{a+}^{\alpha,p} \cap \mathrm{AC}_{b-}^{\alpha,p}$ holds provided that $1 \leq p < \frac{1}{\alpha}$.
\end{proposition}

\begin{proof}
Let $q \in \mathrm{AC}$ given by \eqref{eqabscontclass}. Then
\begin{eqnarray*}
(I^{1-\alpha}_{a+} q) (t) & = & \dfrac{c}{(1-\alpha)\Gamma (1-\alpha)} (t-a)^{1-\alpha} + I^{1}_{a+} ( I^{1-\alpha}_{a+} \varphi ) (t), \\
& = & I^{1}_{a+} \left( \dfrac{c}{\Gamma (1-\alpha) (\cdot-a)^{\alpha}} +   I^{1-\alpha}_{a+} \varphi  \right) (t),
\end{eqnarray*}
for a.e. $t \in [a,b]$. Since $\varphi \in \mathrm{L}^1$, $I^{1-\alpha}_{a+} \varphi  \in \mathrm{L}^p$ for any $1 \leq p < \frac{1}{\alpha}$, see Proposition~\ref{nose}. Note that the function $\frac{c}{\Gamma (1-\alpha) (\cdot-a)^{\alpha}}$ also belongs to $\mathrm{L}^p$ for any $1 \leq p < \frac{1}{\alpha}$. Then, $q$ possesses a left-sided fractional Riemann-Liouville derivative of order $\alpha \in (0,1)$ given by $ D_{a+}^{\alpha}q = \frac{c}{\Gamma (1-\alpha) (\cdot-a)^{\alpha}} + I^{1-\alpha}_{a+} \varphi $. Hence, $D_{a+}^{\alpha}q  \in \mathrm{L}^p$ for any $1 \leq p < \frac{1}{\alpha}$. A similar proof states that $D_{b-}^{\alpha} q  \in \mathrm{L}^p$ for any $1 \leq p < \frac{1}{\alpha}$.
\end{proof}

The following proposition gives another example of explicit functions that belong to $\mathrm{AC}_{a+}^{\alpha,p}$ and $\mathrm{AC}_{b-}^{\alpha,p}$ for some $1 \leq p \leq \infty$ satisfying $0 \leq \frac{1}{p} < \alpha < 1$.

\begin{proposition}\label{propabsdansabs2}
Let us consider:
$$ \mathrm{F} = \{ q \in \mathrm{AC} \; | \; q(a) =0 \text{ and } q' \in \mathrm{L}^r \text{ for some }  1< r \leq \infty \}. $$
Then the inclusion
$$  \mathrm{F} \subset \bigcup_{\substack{1 \leq p \leq \infty \\ 0 \leq \frac{1}{p} < \alpha}}  \mathrm{AC}_{a+}^{\alpha,p} $$
holds true for any $\alpha \in (0,1)$.
\end{proposition}

\begin{proof}
Let $\alpha \in (0,1)$ and let $q \in \mathrm{F}$. From \eqref{eqabscontclass}, $q = I^1_{a+} \varphi$ with $\varphi \in \mathrm{L}^r$ for some $r > 1$. Then $q = I^{\alpha}_{a+} (I^{1-\alpha}_{a+} \varphi)$ and $q$ possesses a left-sided Riemann-Liouville derivative given by $D^{\alpha}_{a+} q = I^{1-\alpha}_{a+}  \varphi$, see Remark~\ref{rmk111}. In the case where $0 < 1-\alpha < \frac{1}{r} < 1$, Proposition~\ref{nose} gives $D^{\alpha}_{a+} q \in \mathrm{L}^{p}$ with $p=\frac{r}{1-(1-\alpha)r}$ satisfying $\frac{1}{p} < \alpha$. In the case where $\frac{1}{r} = 1 - \alpha$ then Proposition~\ref{nose} gives $D^{\alpha}_{a+} q \in \mathrm{L}^{p}$ for every $1 \leq p < \infty$, so in particular for some $p$ satisfying $\frac{1}{p} < \alpha$. In the case where $0 \leq \frac{1}{r} < 1-\alpha < 1$, Proposition~\ref{nose} gives $D^{\alpha}_{a+} q \in \mathrm{L}^{p}$ with $p=\infty$ satisfying $\frac{1}{p} < \alpha$. As a conclusion, in all cases the inclusion $\mathrm{F} \subset \mathrm{AC}_{a+}^{\alpha,p}$ is proved for some $1 \leq p \leq \infty$ such that $0 \leq \frac{1}{p} < \alpha < 1$.
\end{proof}

\begin{remark}
The right-sided analogous proposition can be similarly derived.
\end{remark}

In what follows, we are specially interested in the case $0 \leq \frac{1}{p} < \alpha < 1$. In such a case, for any $\varphi \in \mathrm{L}^p$, $I^\alpha_{a+} \varphi$ can be identified to a continuous function vanishing at $t=a$, see Proposition~\ref{nose}. Similarly, $I^\alpha_{b-} \psi$ can be identified to a continuous function vanishing at $t=b$ for any $\psi \in \mathrm{L}^p$.

In the sequel, in the case $0 \leq \frac{1}{p} < \alpha < 1$, we consider the following pointwise identification for every $q \in  \mathrm{AC}_{a+}^{\alpha,p}$:
\begin{equation}\label{identificationleft}
q(t) = \frac{1}{\Gamma(\alpha)}\frac{c}{(t-a)^{1-\alpha}} + (I^{\alpha}_{a+} \varphi )(t) ,\quad t\in (a,b],
\end{equation}
where $ \varphi = D^\alpha_{a+} q \in \mathrm{L}^p$ and $c = (I^{1-\alpha}_{a+} q)(a)$. Hence, $q$ is well-defined at $t=b$. The analogous pointwise identification is considered for every $q \in  \mathrm{AC}_{b-}^{\alpha,p}$:
\begin{equation}\label{identificationright}
q(t) = \frac{1}{\Gamma(\alpha)}\frac{d}{(b-t)^{1-\alpha}} + (I^{\alpha}_{b-} \psi )(t) ,\quad t\in [a,b),
\end{equation}
where $ \psi = D^\alpha_{b-} q \in \mathrm{L}^p$ and $d = (I^{1-\alpha}_{b-} q)(b)$. Hence, $q$ is well-defined at $t=a$.

\subsection{Main result}\label{section33}

The following theorem on the integration by parts for fractional Riemann-Liouville derivatives holds.

\begin{theorem}[fractional integration by parts formula]\label{cze3}
If $0 \leq \frac{1}{p} < \alpha < 1$ and $0 \leq \frac{1}{r} < \alpha < 1$, then
\begin{multline}\label{eq5555}
\int_a^b (D^{\alpha}_{a+} q_1) (\t) \cdot q_2(\t) \; d\t = \int_a^b q_1(\t) \cdot D^\alpha_{b-} q_2(\t) \; d\t \\ + q_1(b) \cdot (I_{b-}^{1-\alpha} q_2)(b)  -(I_{a+}^{1-\alpha} q_1)(a) \cdot q_2(a)
\end{multline}
for any $q_1 \in \mathrm{AC}_{a+}^{\alpha,p}$, $q_2 \in \mathrm{AC}_{b-}^{\alpha,r}$.
\end{theorem}

\begin{proof}
Let $q_1$, $q_2$ be given by \eqref{identificationleft}, \eqref{identificationright} with $c\in\mathbb{R}^{m}$, $\varphi\in \mathrm{L}^{p}$ and $d\in\mathbb{R}^{m}$, $\psi\in
\mathrm{L}^{r}$, respectively. The assumption $0 \leq \frac{1}{p} < \alpha < 1$ ensures that $1 \leq p^\prime < \frac{1}{1-\alpha}$. Then, $ \mathrm{AC}_{b-}^{\alpha,r} \subset \mathrm{L}^{p^\prime}$ (see Propositon~\ref{remlr}) and $\int_a^b (D^{\alpha}_{a+} q_1) (\t) \cdot q_2(\t) d\t$ is well-defined. Proposition~\ref{czint} gives
\begin{eqnarray*}
\int_a^b (D^{\alpha}_{a+} q_1) (\t) \cdot q_2(\t) \; d\t & = & \int_a^b \varphi(\t) \cdot q_2 (\t) \; d\t \\
& = & \dfrac{d}{\Gamma(\alpha)} \cdot \int_a^b \frac{\varphi(\t)}{(b-\t)^{1-\alpha}} \; d\t + \int_a^b \varphi(\t) \cdot (I^\alpha_{b-} \psi) (\t) \; d\t \\
& = & (I^\alpha_{a+} \varphi) (b) \cdot d + \int_a^b (I^\alpha_{a+} \varphi)(\t) \cdot \psi (\t) \; d\t.
\end{eqnarray*}
Similarly (but without using Proposition~\ref{czint}), it holds 
\begin{equation*}
\int_a^b q_1 (\t) \cdot (D^{\alpha}_{b-} q_2) (\t) \; d\t = c \cdot (I^{\alpha}_{b-} \psi)(a) + \int_a^b (I^\alpha_{a+} \varphi)(\t) \cdot \psi (\t) \; d\t.
\end{equation*}
The following equality follows:
\begin{equation*}
\int_a^b (D^{\alpha}_{a+} q_1) (\t) \cdot q_2(\t)- q_1 (\t) \cdot (D^{\alpha}_{b-} q_2) (\t) \; d\t = (I^\alpha_{a+} \varphi) (b) \cdot d - c \cdot (I^{\alpha}_{b-} \psi)(a).
\end{equation*}
Equalities $(I^\alpha_{a+} \varphi) (b) = q_1 (b) - \frac{1}{\Gamma (\alpha)} \frac{c}{(b-a)^{1-\alpha}}$, $(I^\alpha_{a+} \psi) (a) = q_2 (a) - \frac{1}{\Gamma (\alpha)} \frac{d}{(b-a)^{1-\alpha}}$, $c = (I^{1-\alpha}_{a+} q_1) (a)$ and $d = (I^{1-\alpha}_{b-} q_2) (b)$ conclude the proof.
\end{proof}

\begin{remark}
In the limit case of $\alpha=p=q=1$, the above theorem leads to the
classical integration by parts formula given by:
$$ \displaystyle\int_{a}^{b} q'_1(\t) \cdot q_2(\t) \; d\t= q_1(b) \cdot q_2(b) - q_1(a) \cdot q_2(a)- \int_{a}^{b} q_1(\t) \cdot q'_2(\t) \; d\t $$
for any $q_1$, $q_2\in \mathrm{AC}$.
\end{remark}

\section{Application to critical points of Bolza funtionals}\label{section4}

In this section we use the fractional fundamental lemma (Theorem~\ref{GDBR}) and the fractional integration by parts formula (Theorem~\ref{cze3}) to investigate the critical points of fractional Bolza functionals. We prove that such points satisfy an Euler-Lagrange equation with appropriate boundary conditions.

In the whole section, we fix $\alpha\in(0,1)$ and $1 \leq p \leq \infty$ and we assume that $0 \leq \frac{1}{p} < \alpha < 1$. Let us consider the following Bolza functional%
$$ \fonction{\Phi}{\mathrm{AC}^{\alpha,p}_{a+}}{\R}{q}{\displaystyle \int_a^b L(\t,q(\t),(D^\alpha_{a+} q) (\t) ) \; d\t + \ell ((I^{1-\alpha}_{a+} q)(a),q(b) )}$$
where $L$ is a Lagrangian, \textit{i.e.} a continuous mapping of class $\mathscr{C}^1$ in its two last variables
$$ \fonction{L}{[a,b] \times \R^m \times \R^m}{\R}{(t,x,v)}{L(t,x,v),} $$
and $\ell$ is a mapping of class $\mathscr{C}^1$
$$ \fonction{\ell}{\R^m \times \R^m}{\R}{(x_1,x_2)}{\ell (x_1,x_2).}$$
Recall that the first variation $\delta\Phi(q,h)$ of the functional $\Phi$ at the point $q \in \mathrm{AC}_{a+}^{\alpha,p}$ in the direction $h \in \mathrm{AC}_{a+}^{\alpha,p}$ is defined by (see \cite{ATF})
$$ \delta\Phi(q,h) :=\underset{\lambda \to 0}{\lim}\frac{\Phi(q+\lambda
h)-\Phi(q)}{\lambda}. $$
By a critical point of $\Phi$ we mean a point $q \in \mathrm{AC}_{a+}^{\alpha,p}$ such that $
\delta\Phi(q,h)=0$ for any $h\in \mathrm{AC}_{a+}^{\alpha,p}$.

In Section~\ref{section41} we introduce the notion of \textit{quasi-polynomially controlled growth} for $L$. In Section~\ref{section42}, under this assumption, we prove that $\Phi$ admits a first variation and the main result is next derived.

\subsection{Quasi-polynomially controlled growth for Lagrangian $L$}\label{section41}
Let us introduce a set $\mathscr{P}^{\alpha,p}_M$ containing quasi-polynomial functions.

\begin{definition}\label{defdef5}
Let $1 \leq M \leq \infty$. By $\mathscr{P}^{\alpha,p}_M$ we denote the set of all quasi-polynomial functions 
$$ P : (t,x,v) \in  [a,b] \times \R^m \times \R^m \longmapsto P(t,x,v) \in \R^+ $$
that can be written as follows
\begin{center}
\begin{tabular}{|c|c|c|}\hline
 & $p < \infty$ & $p = \infty$ \\ \hline 
$M < \infty$ & \parbox{6.9cm}{ $$P(t,x,v) = \displaystyle \sum_{k=0}^N c_k(t) \Vert x \Vert^{s_{1,k}} \Vert v \Vert^{s_{2,k}}$$ with \begin{equation*}
\left\lbrace \begin{array}{lcl}
\frac{s_{2,k}}{p} \leq \frac{1}{M} & \text{if} & s_{1,k} = 0 \\ \\
(1-\alpha)s_{1,k} + \frac{s_{2,k}}{p} < \frac{1}{M} & \text{if} & s_{1,k} > 0
\end{array} \right.
\end{equation*} } & \parbox{5.26cm}{ $$P(t,x,v) = \displaystyle \sum_{k=0}^N c_k(t,v) \Vert x \Vert^{s_{1,k}} $$  with $$(1-\alpha) s_{1,k} < \frac{1}{M}$$ } \\ \hline
$M= \infty$ & \parbox{6.9cm}{ $$ P(t,x,v) = c_0 (t) $$ } & \parbox{5.26cm}{ $$ P(t,x,v) = c_0 (t,v) $$ } \\ \hline
\end{tabular}
\end{center}
where $N \in \mathbb{N}$, $c_k$ are continuous mappings with nonnegative values and $s_{1,k}, s_{2,k}$ are nonnegative reals.
\end{definition}
The interest of these quasi-polynomial functions is given in the following lemma.

\begin{lemma}\label{lempolynomial}
For any $1 \leq M \leq \infty$ and any $P \in \mathscr{P}^{\alpha,p}_M$, we have
$$ P(\cdot,q,D^\alpha_{a+} q) \in \mathrm{L}^M$$
for any $q \in \mathrm{AC}^{\alpha,p}_{a+}$.
\end{lemma}

\begin{proof}
Since $q \in \mathrm{L}^r$ for every $1 \leq r < \frac{1}{1-\alpha}$ (see Proposition~\ref{remlr}) and $D^\alpha_{a+} q \in \mathrm{L}^p$ for any $ q \in \mathrm{AC}^{\alpha,p}_{a+}$, one can easily obtain this result from H\"older's inequalities.
\end{proof}

We introduce the following notion.

\begin{definition}\label{defdef6}
The growth of $L$ is said to be quasi-polynomially controlled if there exist $P_0 \in \mathscr{P}^{\alpha,p}_1$, $\frac{1}{\alpha} < s \leq \infty$, $P_1 \in \mathscr{P}^{\alpha,p}_s$ and $P_2 \in \mathscr{P}^{\alpha,p}_{p'}$ such that
\begin{enumerate}
\item $\vert L(t,x,v) \vert \leq P_0(t,x,v)$;
\item $\Vert L_x(t,x,v) \Vert \leq P_1(t,x,v)$;
\item $\Vert L_v(t,x,v) \Vert \leq P_2(t,x,v)$;
\end{enumerate}
for any $(t,x,v) \in [a,b] \times \R^m \times \R^m$.
\end{definition}
It can be noted that the more $\alpha$ is close to $1$ and/or the more $p$ is large, then the less the quasi-polynomially controlled growth of $L$ is a restrictive assumption.

\subsection{Main result}\label{section42}
The quasi-polynomially controlled growth of $L$ ensures the existence of the first variation of $\Phi$.

\begin{lemma}\label{diff}
If $L$ has a quasi-polynomially controlled growth, then the functional $\Phi$ is well-defined on $\mathrm{AC}^{\alpha,p}_{a+}$ and has the first variation $\delta\Phi(q,h)$ at any point
$q\in \mathrm{AC}_{a+}^{\alpha,p}$ and in any direction $h\in \mathrm{AC}_{a+}^{\alpha,p}$,
given by
\begin{multline*}
\delta \Phi (q,h) = \int_a^b L_x (\t,q(\t),(D^\alpha_{a+} q)(\t) ) \cdot h(\t) + L_v (\t,q(\t),(D^\alpha_{a+} q)(\t) ) \cdot (D^\alpha_{a+} h)(\t) \; d\t \\ + \ell_{x_1} ((I^{1-\alpha}_{a+} q)(a),q(b) ) \cdot (I^{1-\alpha}_{a+} h)(a) + \ell_{x_2} ((I^{1-\alpha}_{a+} q)(a),q(b) ) \cdot h(b).
\end{multline*}
\end{lemma}

\begin{proof}
The first variation of the mapping $q \in \mathrm{AC}^{\alpha,p}_{a+} \longmapsto \ell ((I^{1-\alpha}_{a+} q)(a),q(b) ) \in \R$ obviously exists. Then, we assume that $\ell = 0$ in the sequel of the proof. Well-definedness of $\Phi$ and $\delta\Phi$ follows from the quasi-polynomially controlled growth of $L$ and from Lemma~\ref{lempolynomial}. Indeed, for any $q \in \mathrm{AC}^{\alpha,p}_{a+}$, one has:
\begin{enumerate}
\item $L(\cdot,q,D^\alpha_{a+} q) \in \mathrm{L}^1$;
\item $L_x (\cdot,q,D^\alpha_{a+} q) \in \mathrm{L}^s$ and $h \in \mathrm{L}^{s'}$ since $\frac{1}{\alpha} < s \leq \infty$ and $h \in \mathrm{L}^r$ for every $1 \leq r < \frac{1}{1-\alpha}$;
\item $L_v (\cdot,q,D^\alpha_{a+} q) \in \mathrm{L}^{p'}$ and $D^\alpha_{a+} h \in \mathrm{L}^{p}$.
\end{enumerate}
Let us write $\Phi$ as the superposition $ \Phi=\Phi_{2}\circ\Phi_{1} $ where
$$ \Phi_{1}:  q \in \mathrm{AC}_{a+}^{\alpha,p} \longmapsto L(\cdot,q,D_{a+}^{\alpha} q) \in \mathrm{L}^{1}, $$
$$ \Phi_{2}:q \in \mathrm{L}^1 \longmapsto \displaystyle\int_{a}^{b} q(\t) \; d\t \in\mathbb{R}. $$
Of course, linear and continuous mapping $\Phi_{2}$ is Frechet differentiable
on $\mathrm{L}^{1}$ and the differential at any point $q\in \mathrm{L}^{1}$ is equal to $\Phi
_{2}$. So (see \cite[section 2.2.2]{ATF}), it is sufficient to show that
$\Phi_{1}$ has the first variation $\delta\Phi_{1}(q,h)$ at any $q\in
\mathrm{AC}_{a+}^{\alpha,p}$ and in any direction $h\in \mathrm{AC}_{a+}^{\alpha,p}$, given by
$$ \delta\Phi_{1}(q,h)=L_{x}(\cdot,q,D_{a+}^{\alpha}q) \cdot h + L_{v}(\cdot,q,D_{a+}^{\alpha}q) \cdot D_{a+}^{\alpha} h. $$
Indeed, it can be checked with the aid of the Lebesgue dominated convergence
theorem and the Mean value theorem.
\end{proof}

Now, we shall prove that critical points of $\Phi$ are solutions to a
boundary value problem.

\begin{theorem}\label{TEL1}
If $L$ has a quasi-polynomially controlled growth and if $q\in
\mathrm{AC}_{a+}^{\alpha,p}$ is a critical point of $\Phi$, then $L_{v}(\cdot
,q,D_{a+}^{\alpha}q) \in \mathrm{AC}_{b-}^{\alpha,s}$ and%
\begin{equation}\label{cztery}
(D_{b-}^{\alpha} L_{v}(\cdot,q,D_{a+}^{\alpha}q) )(t)=-L_{x}(t,q(t),(D_{a+}^{\alpha}q)(t)),\quad t\in [a,b]\text{ a.e.,} 
\end{equation}
\begin{equation}\label{piec}
L_{v}(a,q(a),(D_{a+}^{\alpha}q)(a))= \ell_{x_1} ((I^{1-\alpha}_{a+} q)(a),q(b) ) , 
\end{equation}
\begin{equation}\label{szesc}
( I_{b-}^{1-\alpha} L_{v}(\cdot,q,D_{a+}^{\alpha}q) ) (b)= - \ell_{x_2} ((I^{1-\alpha}_{a+} q)(a),q(b) ).
\end{equation}
Equation~\eqref{cztery} is the so-called Euler-Lagrange equation. Equalities~\eqref{piec} and \eqref{szesc} are boundary conditions.
\end{theorem}

\begin{proof}
Let $q \in \mathrm{AC}_{a+}^{\alpha,p}$ be a critical point of $\Phi$. Then Lemma~\ref{diff} leads to
\begin{multline}\label{eq9874}
\int_a^b L_x (\t,q(\t),(D^\alpha_{a+} q)(\t) ) \cdot h(\t) + L_v (\t,q(\t),(D^\alpha_{a+} q)(\t) ) \cdot (D^\alpha_{a+} h)(\t) \; d\t \\ + \ell_{x_1} ((I^{1-\alpha}_{a+} q)(a),q(b) ) \cdot (I^{1-\alpha}_{a+} h)(a) + \ell_{x_2} ((I^{1-\alpha}_{a+} q)(a),q(b) ) \cdot h(b) = 0
\end{multline}
for any $h \in \mathrm{AC}_{a+}^{\alpha,p}$.

Note that $(I^{1-\alpha}_{a+} h)(a) = h(b) = 0$ for any $h\in \mathscr{C}^\infty_c \subset \mathrm{AC}_{a+}^{\alpha,p}$. Then we have
\begin{equation*}
\displaystyle \int_{a}^{b} L_{x}(\t,q(\t),(D_{a+}^{\alpha}q)(\t)) \cdot h(\t)+L_{v}(\t,q(\t),(D_{a+}^{\alpha}q)(\t)) \cdot (D_{a+}^{\alpha}h)(\t) \; d\t = 0, 
\end{equation*}
for any $h\in \mathscr{C}^\infty_c$. Theorem~\ref{GDBR} implies that $L_{v}(\cdot,q,D_{a+}^{\alpha}q)\in \mathrm{AC}_{b-}^{\alpha,s}$ and Equation~\eqref{cztery} holds true almost everywhere.

Let us observe that the boundary conditions \eqref{piec}, \eqref{szesc} are satisfied. Indeed, from \eqref{cztery} and integrating by parts (see Theorem \ref{cze3}), we obtain
\begin{multline*}
\int_a^b L_{v}(\t,q(\t),(D_{a+}^{\alpha}q)(\t)) \cdot (D_{a+}^{\alpha}h)(\t) \; d\t =
- \int_a^b L_{x}(\t,q(\t),(D_{a+}^{\alpha}q)(\t)) \cdot h(\t) \; d\t \\
+( I_{b-}^{1-\alpha} L_{v}(\cdot,q,D_{a+}^{\alpha}q ))(b) \cdot h(b)- L_{v}(a,q(a),(D_{a+}^{\alpha}q)(a)) \cdot (I_{a+}^{1-\alpha}h)(a),
\end{multline*}
for any $h\in \mathrm{AC}_{a+}^{\alpha,p}$. Therefore, Equality~\eqref{eq9874} becomes
\begin{multline}\label{equseful}
\Big[\ell_{x_2} ((I^{1-\alpha}_{a+} q)(a),q(b)) +( I_{b-}^{1-\alpha} L_{v}(\cdot,q,D_{a+}^{\alpha}q ))(b) \Big] \cdot h(b) \\ + \Big[\ell_{x_1} ((I^{1-\alpha}_{a+} q)(a),q(b) ) - L_{v}(a,q(a),(D_{a+}^{\alpha}q)(a)) \Big] \cdot (I_{a+}^{1-\alpha}h)(a) = 0
\end{multline}
for any $h\in \mathrm{AC}_{a+}^{\alpha,p}$.

Now, let us fix any $i\in\{1,\ldots,n\}$ and let us consider the function $ h(t)=(0,\ldots,0,(I_{a+}^{\alpha}1)(t),0,\ldots,0) $
with the nonzero $i$-th coordinate function. It belongs to $\mathrm{AC}_{a+}^{\alpha
,p}$, $h(b)=(0,\ldots,0,(I_{a+}^{\alpha
}1)(b),0,\ldots,0)$ and Proposition~\ref{dziewiec} gives $(I_{a+}^{1-\alpha}h)(a)=0$. Since $i\in\{1,\ldots,n\}$ is arbitrary and $(I_{a+}^{\alpha
}1)(b)>0$, \eqref{equseful} implies that
$$ ( I_{b-}^{1-\alpha} L_{v}(\cdot,q,D_{a+}^{\alpha}q ))(b) = -\ell_{x_2} ((I^{1-\alpha}_{a+} q)(a),q(b)). $$
Let us consider the function
$$ h(t)=(0,\ldots,0,\frac{1}{\Gamma(\alpha)}\frac{1}{(t-a)^{1-\alpha}} +(I_{a+}^{\alpha} \theta)(t),0,\ldots,0), $$
with $\theta = \frac{-\Gamma(\alpha+1)}{\Gamma(\alpha)(b-a)}$. It belongs to $\mathrm{AC}^{\alpha,p}_{a+}$, Proposition~\ref{dziewiec} gives
$$(I_{a+}^{1-\alpha}h)(a)=(0,\ldots,0,1,0,\ldots,0)$$
and the value of $\theta$ ensures that $h(b)=0$. Thus, \eqref{equseful} implies that
$$ L_{v}(a,q(a),(D_{a+}^{\alpha}q)(a))= \ell_{x_1} ((I^{1-\alpha}_{a+} q)(a),q(b) ). $$
The proof is complete.
\end{proof}

\begin{example}
Let $1 < r < \infty$ and let $L$ be given by $L(t,x,v) = \Vert x \Vert^r + \Vert v \Vert^r$. Then, $L$ has a quasi-polynomially controlled growth for any $\alpha \in (\frac{1}{r'},1)$ and any $r \leq p \leq \infty$. Consequently, Theorem~\ref{TEL1} can be applied for any $\alpha \in (\frac{1}{r'},1)$ and any $r \leq p \leq \infty$.

For example, in the case $r=2$ and $\alpha \in (\frac{1}{2},1)$, the Euler-Lagrange equation associated with $\Phi$ is given by
$$ D^\alpha_{b-}  D^\alpha_{a+} q = - q .$$
Moreover, if $\ell = 0$, the boundary conditions are given by
$$ (D^\alpha_{a+} q) (a) = 0, $$
$$ (I^{1-\alpha}_{b-}  D^\alpha_{a+} q)(b) = 0 .$$
\end{example}

\section{Application to linear boundary value problems}\label{section5}

The aim of this section is to investigate an existence result for solutions of the following linear boundary value problem
\begin{equation}\label{10a}
( D^\alpha_{b-}  D^\alpha_{a+} q)(t)+q(t) = f(t), \ t\in [a,b]\text{a.e.},
\end{equation}
\begin{equation}\label{10b}
(I^{1-\alpha}_{a+} q)(a) = q_a, \ q(b) = q_b, 
\end{equation}
where $\alpha \in (\frac{1}{2},1)$, $f \in \mathrm{L}^2$ and $q_a$, $q_b \in \R^m$. Our strategy is to use a Hilbert structure of $\mathrm{AC}^{\alpha,2}_{a+}$ and the classical Stampacchia theorem (see \cite{Bre}) recalled below. 

\begin{proposition}[Stampacchia theorem]
Let $(H,\Vert \cdot \Vert_H)$ be a real Hilbert space and $K\subset H$ be a nonempty closed convex set. Let $a:H\times H\longrightarrow\mathbb{R}$ be a continuous and coercive bilinear form and let $\phi : H \longrightarrow \R$ be a continuous linear form. Then, there exists exactly one point $x\in K$ such that $a(x,y-x)\geq \phi(y-x)$ for any $ y\in K$.
\end{proposition}
In what follows, we assume that $p=2$ and $\alpha\in(\frac{1}{2},1)$. In particular, $\mathrm{AC}^{\alpha,2}_{a+} \subset \mathrm{L}^2$. We endow $\mathrm{AC}^{\alpha,2}_{a+}$ with the scalar product
$$ \langle q_1, q_2 \rangle_{\mathrm{AC}^{\alpha,2}_{a+}} = \int_a^b q_1(\t) \cdot q_2(\t) \; d\t + \int_a^b (D^\alpha_{a+} q_1)(\t) \cdot (D^\alpha_{a+} q_2)(\t) \; d\t $$
and by $\Vert \cdot \Vert_{\mathrm{AC}^{\alpha,2}_{a+}}$ we denote the generated norm, that is:
$$ \Vert q \Vert_{\mathrm{AC}^{\alpha,2}_{a+}} = (\Vert q \Vert^2_{\mathrm{L}^2} + \Vert D^\alpha_{a+} q \Vert^2_{\mathrm{L}^2} )^{1/2}, \ q \in \mathrm{AC}^{\alpha,2}_{a+}. $$

\begin{lemma}
The functional space $(\mathrm{AC}^{\alpha,2}_{a+}, \Vert \cdot \Vert_{\mathrm{AC}^{\alpha,2}_{a+}} )$ is a Hilbert space.
\end{lemma} 

\begin{proof}
Let us prove that $\mathrm{AC}^{\alpha,2}_{a+}$ is complete. Let $(q_n)_{n \in \N} \subset \mathrm{AC}^{\alpha,2}_{a+}$ be a Cauchy sequence. Then $(q_n)_{n \in \N}$ and $(D^\alpha_{a+} q_n)_{n \in \N}$ are Cauchy sequences in the complete space $(\mathrm{L}^2,\Vert \cdot \Vert_{\mathrm{L}^2} )$. By $q$ and $u$ we denote their respective limits in $(\mathrm{L}^2,\Vert \cdot \Vert_{\mathrm{L}^2} )$. Our aim is to prove that $q \in \mathrm{AC}^{\alpha,2}_{a+}$ and $D^\alpha_{a+} q = u$. Theorem~\ref{cze3} leads to
$$ \int_a^b (D^\alpha_{a+} q_n)(\t) \cdot h(\t) - q_n \cdot (D^{\alpha}_{b-} h)(\t) \; d\t = 0 $$
for any $n \in \N$ and any $h \in \mathscr{C}^\infty_c$. Passing to the limit on $n$, we obtain
$$ \int_a^b u(\t) \cdot h(\t) - q \cdot (D^{\alpha}_{b-} h)(\t) \; d\t = 0 $$
for any $h \in \mathscr{C}^\infty_c$. The counterpart of Theorem~\ref{GDBR} for the right-sided fractional derivative concludes the proof.
\end{proof}

\begin{lemma}
The set
$$ K=\{q\in \mathrm{AC}_{a+}^{\alpha,2};\ (I_{a+}^{1-\alpha}q)(a)=q_a,\ q(b)=q_b\} $$
is a nonempty closed and convex subset of $\mathrm{AC}_{a+}^{\alpha,2}$.
\end{lemma}

\begin{proof}
The convexity of $K$ is obvious. To prove that $K$ is nonempty, it is sufficient to consider
$$ q(t) = \dfrac{1}{\Gamma (\alpha)} \dfrac{q_a}{(t-a)^{1-\alpha}}+ (I^\alpha_{a+} \theta)(t), \ t \in (a,b] $$
with $\theta = \frac{\Gamma(\alpha+1)}{(b-a)^\alpha} q_b - \frac{\Gamma(\alpha+1)}{\Gamma(\alpha)(b-a)} q_a$. Let us prove that $K$ is closed. Let $(q_n)_{n \in \N} $ be a sequence of $K$ tending to $q$ in $\mathrm{AC}^{\alpha,2}_{a+}$. Our aim is to prove that $(I_{a+}^{1-\alpha}q)(a)=q_a$ and $q(b)=q_b$. Since $(q_n)_{n \in \N}$ tends to $q$ in $(\mathrm{L}^2,\Vert \cdot \Vert_{\mathrm{L}^2} )$, we have $(I^{1-\alpha}_{a+} q_n)_{n \in \N}$ tends to $I^{1-\alpha}_{a+} q$ in $(\mathrm{L}^2,\Vert \cdot \Vert_{\mathrm{L}^2} )$. Since $(D^{\alpha}_{a+} q_n )_{n \in \N} = (\frac{d}{dt} (I^{1-\alpha}_{a+} q_n) )_{n \in \N}$ tends to $D^{\alpha}_{a+} q = \frac{d}{dt} (I^{1-\alpha}_{a+} q)$ in $(\mathrm{L}^2,\Vert \cdot \Vert_{\mathrm{L}^2} )$, we conclude that $(I^{1-\alpha}_{a+} q_n)_{n \in \N}$ tends to $I^{1-\alpha}_{a+} q$ in $(\mathrm{W}^{1,2},\Vert \cdot \Vert_{\mathrm{W}^{1,2}} )$. From the compact embedding of $(\mathrm{W}^{1,2},\Vert \cdot \Vert_{\mathrm{W}^{1,2}} )$ in the set of continuous functions endowed with the usual uniform norm (see \cite{Bre}), we conclude that $((I^{1-\alpha}_{a+} q_n) (a))_{n \in \N}$ tends to $(I^{1-\alpha}_{a+} q) (a)$ and then $(I^{1-\alpha}_{a+} q) (a) =q_a$. Finally, the integral representation of $q_n$ at $t=b$ gives
$$ q_b = \dfrac{1}{\Gamma (\alpha)} \dfrac{q_a}{(b-a)^{1-\alpha}} + \dfrac{1}{\Gamma(\alpha)} \int_a^b \dfrac{D^{\alpha}_{a+} q_n(\t)}{(b-\t)^{1-\alpha}} \; d\t $$
for any $n \in \N$. Passing to the limit on $n$ leads to $q(b)=q_b$ and concludes the proof.
\end{proof}

Finally, we prove

\begin{theorem}\label{thmexistence}
The linear boundary value problem \eqref{10a}-\eqref{10b} has a solution
$q \in \mathrm{AC}_{a+}^{\alpha,2}$.
\end{theorem}

\begin{proof}
Let us consider the continuous coercive bilinear form $a:\mathrm{AC}_{a+}^{\alpha
,2}\times \mathrm{AC}_{a+}^{\alpha,2}\longrightarrow\mathbb{R}$ given by
$$ a(q_1,q_2) = \langle q_1,q_2 \rangle_{\mathrm{AC}^{\alpha,2}_{a+}}, \quad q_1,q_2 \in \mathrm{AC}^{\alpha,2}_{a+},$$
the continuous linear form $\phi : \mathrm{AC}^{\alpha,2}_{a+} \longrightarrow \R$ given by
$$ \phi (q) = \int_a^b f(\t) \cdot q(\t) \; d\t, \ q \in \mathrm{AC}^{\alpha,2}_{a+} $$
and the nonempty closed convex set $K$ defined in the previous lemma. The classical Stampacchia theorem gives the existence of a function $q \in K$ such that
\begin{multline}
\int_a^b q(\t) \cdot (q(\t)-q_1(\t)) \; d\t + \int_a^b (D^\alpha_{a+} q)(\t) \cdot ((D^\alpha_{a+} q)(\t)-(D^\alpha_{a+} q_1)(\t)) \; d\t \\ \geq \int_a^b f(\t) \cdot (q(\t)-q_1(\t)) \; d\t
\end{multline}
for any $q_1 \in K$. Considering the above inequality with functions $q_1 = q \pm h\in K$
with $h\in \mathscr{C}^\infty_c$ we assert that%
$$ \int_a^b (q(\t)-f(\t)) \cdot h(\t) + D^\alpha_{a+} q(\t) \cdot D^\alpha_{a+} h(\t) \; d\t = 0 $$
for any $h\in \mathscr{C}^\infty_c$. The fractional fundamental lemma (Theorem~\ref{GDBR}) concludes.
\end{proof}

\section*{Notes and acknowledgement}

The results of this paper (except Section~\ref{section5}) were presented during the 7th International Workshop on Multidimensional (nD) Systems, (Poitiers, France, 2011) and published - without proofs - in conference proceedings (see \cite{Poitiers}) in the special case $\alpha \in (\frac{1}{2},1)$. The aim of this paper is to give detailed proofs of these theorems and to improve them by removing the assumption $\alpha \in (\frac{1}{2},1)$. In Section~\ref{section5}, we add an application of the obtained results to linear fractional boundary value problems. To our best knowledge, the results presented in the paper have not been obtained by other authors. 

Finally, let us mention that \cite{IdMaj} contains the following theorems: a fractional fundamental lemma for $\alpha \in(n-1/2,n)$ with $n\in\mathbb{N}$, $n\geq2$, a theorem on the fractional integration by parts for $\alpha \in(n-1,n)$ with $n\in\mathbb{N}$, $n\geq2$ and $\frac{1}{\alpha-n+1} < p < \infty$, $\frac{1}{\alpha-n+1} < q < \infty$.


The project was partially financed with funds of National Science Centre,
granted on the basis of decision DEC-2011/01/B/ST7/03426.

\end{document}